\documentclass[letterpaper,11pt,reqno]{amsart}
\usepackage[active]{srcltx}
\usepackage{longtable}
\usepackage{amsmath,amssymb,amscd, booktabs,mathtools}
\usepackage[all]{xy}
\usepackage[headings]{fullpage}
\newcommand{\mat}[4]{{\setlength{\arraycolsep}{0.5mm}\left[
\begin{array}{cc}#1&#2\\#3&#4\end{array}\right]}}
\newcommand{\forget}[1]{}

\allowdisplaybreaks

\newtheorem{lemma}{Lemma}[section]
\newtheorem{theorem}[lemma]{Theorem}

\newtheorem*{maintheorem}{Main Theorem}
\newtheorem*{theorem*}{Theorem ($\GL(2)$ twisting)}

\usepackage[pdftex, colorlinks, linkcolor={blue}, citecolor={blue} ,pdfpagelabels, 
hyperindex, plainpages, pageanchor ]{hyperref}

\newcommand{\Q}{{\mathbb Q}}

\newcommand{\C}{{\mathbb C}}

\newcommand{\p}{\mathfrak p}
\newcommand{\OF}{{\mathfrak o}}
\newcommand{\GL}{{\rm GL}}

\newcommand{\SL}{{\rm SL}}

\newcommand{\GSp}{{\rm GSp}}

\newcommand{\K}[1]{{\rm K}(\p^{#1})}
\newcommand{\Kl}[1]{{\rm Kl}(\p^{#1})}

\newcommand{\vl}{{\rm vol}}

\newcommand{\diag}{{\rm diag}}

\begin{document}
 
\title{Twisting of paramodular vectors}
\author[Johnson-Leung and Roberts]{Jennifer Johnson-Leung\\
Brooks Roberts}
\begin{abstract}
Let $F$ be a non-archimedean local field of characteristic zero, let
$(\pi,V)$ be an irreducible, admissible representation of $\GSp(4,F)$
with trivial central character, and let $\chi$ be a  quadratic
character of $F^\times$ with conductor $c(\chi)>1$. We define a  twisting operator $T_\chi$ from
paramodular vectors for $\pi$ of level $n$ to paramodular vectors for 
$\chi \otimes \pi$ of level $\max(n+2c(\chi),4c(\chi))$, and prove that 
this operator has properties analogous to the well-known
$\GL(2)$ twisting operator. 
\end{abstract}
\maketitle

\section{Introduction}
Let $k$ and $M$ be positive integers, and let $\chi$ be a quadratic Dirichlet character mod $C$. If $f \in S_k(\Gamma_0(M))$ is a cusp form of weight $k$
with respect to $\Gamma_0(M)$ with Fourier expansion
$$
f(z) = \sum_{m=1}^\infty a(m) e^{2\pi im z},
$$
then the twist $f_\chi$ of $f$ by $\chi$ is the element of $S_k(\Gamma_0(MC^2))$ with Fourier expansion
$$
f_\chi (z) = \sum_{m=1}^\infty \chi(m) a(m) e^{2\pi im z}. 
$$
See, for example, Proposition 3.64 of \cite{S}.
In fact, twisting of cusp forms is a local operation when  cusp forms are identified as  automorphic forms on the adeles of $\GL(2)$ over $\Q$.  

Let $F$ be a nonarchimedean local field of characteristic zero with ring of integers $\OF$ and maximal ideal $\p$, let $(\pi,V)$ be a smooth representation of $\GL(2,F)$ for which the center of $\GL(2,F)$ acts trivially, and  let $\chi$ be a quadratic character of $F^\times$.  For $n$ a
non-negative integer, we let $V(n)$ and $V(n,\chi)$ be the spaces of $v \in V$ such that $\pi(k)v=v$ and $\pi(k)v = \chi(\det(k)) v$, respectively, for 
$k \in \Gamma_0(\p^n)$; here $\Gamma_0(\p^n)$ is the subgroup of $\GL(2,\OF)$ of elements which are upper triangular mod $\p^n$. For 
$v \in V$, define the $\chi$-twist $T_\chi (v)$ of $v$ as in \eqref{gl2twisingdefeq}. The main result about $\GL(2)$ twisting is summarized
by the following known theorem.  See section \ref{notationsec} for further definitions and section \ref{genus1sec} for a proof. 
\begin{theorem*}
Let $(\pi,V)$ be a smooth representation of $\GL(2,F)$ for which the center of $\GL(2,F)$ acts trivially, and 
let $\chi$ be a quadratic character of $F^\times$ with conductor $c(\chi) > 0$.  Let $n$ be a non-negative integer and 
define $N = \max(n,2 c(\chi))$. If $v \in V(n)$, then $T_\chi(v) \in V(N,\chi)$.
Moreover, assume that $\pi$ is generic, irreducible and admissible with Whittaker model $\mathcal{W}(\pi,\psi)$. Let $W \in V(n)$.
The $\chi$-twisted zeta integral \eqref{gl2twistedzetaeq} of $T_\chi(W)$ is
$$
Z(s,T_\chi(W),\chi) = (1-q^{-1})G(\chi,-c(\chi))W(1). 
$$
For $n \geq N_\pi$, the image of $T_\chi: V(n) \to V(N,\chi)$ is spanned by the non-zero vector $T_\chi( \beta'{}^{n-N_\pi} W_\pi)$,
where $W_\pi$ is a newform for $\pi$. 
\end{theorem*}
The goal of this paper is to construct an analog of quadratic twisting for paramodular vectors in representations of $\GSp(4,F)$ 
with trivial central character.  Let $(\pi,V)$ be a smooth representation of $\GSp(4,F)$ for which the center of $\GSp(4,F)$
acts trivially. Let $V(n)$ and $V(n,\chi)$ be the spaces of $v \in V$ such that $\pi(k)v =v$ and $\pi(k)v = \chi(\lambda (k)) v$, respectively, 
for $k$ in the paramodular subgroup $\K{n}$ of $\GSp(4,F)$ of level $\p^n$. For $v \in V$, we define the $\chi$-twist $T_\chi(v)$
of $v$ as in  \eqref{twisteq}. Our main result is the following theorem. We refer to section \ref{notationsec} for more definitions
and section \ref{genus2sec} for the proof. 

\begin{maintheorem}
Let $(\pi,V)$ be a smooth representation of $\GSp(4,F)$ for which the center of $\GSp(4,F)$ acts trivially, and let $\chi$ be a quadratic character of $F^\times$ with conductor $c(\chi) >0$.  Let $n$ be a non-negative integer and define 
$
N=  \max(n+2c(\chi),4c(\chi)).
$
 If $v \in V (n)$, then $T_\chi(v) \in V(N,\chi)$. 
Moreover, assume that $\pi$ is generic, irreducible and admissible with Whittaker model $\mathcal{W}(\pi,\psi_{c_1,c_2})$ where $c_1,c_2 \in \OF^\times$. If $W \in  V (n)$, then the $\chi$-twisted zeta integral \eqref{twistedzetaeq} of $T_\chi(W)$ is
$$
Z(s, T_\chi(W),\chi) = (q-1)  q^{c(\chi)} \chi(c_2)G(\chi,-c(\chi))^3 W(1 ).
$$
For $n \geq N_\pi$, the image of $T_\chi: V(n) \to V(N,\chi)$ is spanned by the non-zero vector $T_\chi( \theta'{}^{n-N_\pi} W_\pi)$, 
where $W_\pi$ is a newform for $\pi$. 
\end{maintheorem}

In another work we will consider the application of the paramodular twisting operator $T_\chi$ to Siegel modular forms and the resulting Fourier coefficients. 
One reason that Siegel paramodular forms of degree 2  are of interest is their conjectural connection to abelian surfaces over $\Q$. This is discussed in \cite{BK}; see also \cite{PY}. 

We note that the integer $N$ in the Main Theorem is optimal in the following sense. We may identify the space $V(n,\chi)$ with the space $V_{\chi \otimes \pi}(n)$ of $\K{n}$ fixed vectors in the twisted representation $\chi \otimes \pi$.  Then there exist generic, irreducible, and admissible representations $\pi$ such that $N = N_{\chi \otimes \pi}$. For example, if $\pi$ is a type I representation $\chi_1 \times \chi_2 \rtimes \sigma$ with $\chi_1$, $\chi_2$ and $\sigma$ unramified, then $N_{\chi \otimes \pi} = 4 c(\chi)=\max(0+2c(\chi),4c(\chi))$. Further, suppose that $\pi$ is a type X representation $\pi_1 \rtimes \sigma$ with $\pi_1$ having trivial central character, $\sigma$ unramified, and $2c(\chi)<N_{\pi_1}$. Then $N_{\chi \otimes \pi} = N_{\pi} + 2c(\chi) = \max(N_{\pi}+2c(\chi), 4c(\chi))$. It is interesting to observe, as in this last example, that $N_{\chi \otimes \pi} > N_\pi$ no matter how large $N_\pi$ is.

\section{Notation and preliminaries} 
\label{notationsec}
In this paper $F$ is a nonarchimedean local field of characteristic zero, with ring of integers $\OF$ and generator $\varpi$ of the maximal ideal $\p$ of $\OF$. We fix a non-trivial continuous character $\psi$ of $(F,+)$ such that $\psi(\OF)=1$ but $\psi(\p^{-1}) \neq 1$.  We let $q$ be the number of elements of $\OF/\p$ and use the absolute value on $F$ such that $|\varpi | = q^{-1}$. We use the Haar measure on the additive group $F$ that assigns $\OF$ measure $1$ and the Haar measure on the multiplicative group $F^\times$ that assigns $\OF^\times$  measure $1-q^{-1}$. 
 Throughout the paper $\chi$ is a quadratic character of $F^\times$ with conductor $c(\chi)$, i.e., $c(\chi)$ is the smallest non-negative integer $n$ such that $\chi(1+\p^n)=1$, where we take $1+\p^0=\OF^\times$. 

If $n$ is a non-negative integer, then we let $\Gamma_0(\p^n)$ be the subgroup of $\GL(2,\OF)$ of elements which are upper triangular mod $\p^n$; we will also write $\Gamma_0(\p^n)$ for the analogous subgroup of $\SL(2,\OF)$ when there is no risk of confusion. Let $(\pi,V)$ be a smooth representation of $\GL(2,F)$ for which the center of $\GL(2,F)$ acts trivially, and let $n$ be a non-negative integer. The subspace $V(n)$ consists of the vectors in $V$ fixed by $\Gamma_0(\p^n)$ and  $V(n,\chi)$ is the subspace of vectors $v \in V$ such that $\pi(k) v = \chi(\det (k)) v$ for $k \in \Gamma_0(\p^n)$. We define the level raising operators $\beta,\beta':V(n)\to V(n+1)$ and $\beta,\beta':V(n,\chi)\to V(n+1,\chi)$ by $\beta(v) = \pi(\left[\begin{smallmatrix} 1 & \\ & \varpi \end{smallmatrix} \right]) v$ and $\beta' v = v$.  If $\pi$ is generic, irreducible and admissible, then $V(n)$ is non-zero for some $n$; we let $N_\pi$ be the smallest such $n$. The space $V(N_\pi)$ is one-dimensional; if $W_\pi$ is a non-zero element of $V(N_\pi)$ so that $V(N_\pi) = \C \cdot W_\pi$, then we refer to $W_\pi$ as a \emph{newform}.
The space $V(n)$ for $n \geq N_\pi$ is spanned by the vectors $\beta'{}^i\beta{}^j W_\pi$ where $i$ and $j$ are non-negative integers with $i+j = n-N_\pi$. 
 If $W_\pi$ is viewed as an element of the Whittaker model $\mathcal{W}(\pi,\psi)$ of $\pi$, then $W_\pi(1) \neq 0$. As usual, the elements $W$ of $\mathcal{W}(\pi,\psi)$ satisfy $W(\left[\begin{smallmatrix}1&x\\&1\end{smallmatrix}\right]g)=\psi(x)W(g)$ for $x\in F$ and $g\in\GL(2,F)$. See \cite{C} and \cite{D}.

The theory of paramodular newforms is developed in \cite{RS}, and we will use the notation of
 \cite{RS} concerning $\GSp(4,F)$. We recall some necessary definitions and results.
In particular, $\GSp(4,F)$ is the subgroup of $g \in \GL(4,F)$ such that
$$
{}^t g \begin{bmatrix} &&&1 \\ &&1& \\ &-1&& \\ -1&&& \end{bmatrix} g = \lambda(g) \begin{bmatrix} &&&1 \\ &&1& \\ &-1&& \\ -1&&& \end{bmatrix} 
$$
for some $\lambda(g) \in F^\times$. If $n$ is a non-negative integer, we let $\Kl{n}$ (respectively $\K{n}$) be the subgroup of $k \in \GSp(4,F)$ such that $\lambda (k) \in \OF^\times$ and 
$$
k \in 
\begin{bmatrix}
\OF & \OF & \OF & \OF\\
\p^n&\OF & \OF & \OF \\
\p^n & \OF & \OF & \OF \\
\p^n & \p^n & \p^n & \OF
\end{bmatrix}
\quad 
(\text{resp.} \quad 
k \in 
\begin{bmatrix}
\OF & \OF & \OF & \p^{-n}\\
\p^n&\OF & \OF & \OF \\
\p^n & \OF & \OF & \OF \\
\p^n & \p^n & \p^n & \OF
\end{bmatrix}).
$$
The group $\Kl{n}$ is called the Klingen congruence subgroup of level $\p^n$ and $\K{n}$ is called the paramodular subgroup of level $\p^n$.
 For $a,b,c,d \in F^\times$, we set 
$$
\diag(a,b,c,d) = \begin{bmatrix} a &&& \\ &b&& \\ &&c& \\ &&&d \end{bmatrix}.
$$
This element is in $\GSp(4,F)$ if and only if $ad=bc$. 
Let $(\pi,V)$ be a smooth representation of $\GSp(4,F)$ such that the center of $\GSp(4,F)$ acts trivially. If $n$ is a non-negative integer, then $V_{\mathrm{Kl}}(n)$ and $V(n)$ are  the subspaces of vectors fixed by the Klingen congruence subgroup $\Kl{n}$, and paramodular subgroup $\K{n}$, respectively; additionally,  we let $V_{\mathrm{Kl}}(n,\chi)$ and $V(n,\chi)$  be the subspaces of vectors $v$ in $V$ such that $\pi(k) v = \chi(\lambda (k)) v$ for $k \in \Kl{n}$ and $k \in \K{n}$, respectively. 
Also, we define 
\begin{equation}
\label{specialmateq}
\eta=
\begin{bmatrix}
\varpi^{-1}&&&\\
&1&& \\
&&1& \\
&&&\varpi
\end{bmatrix},
\quad
\tau =
 \begin{bmatrix}
1&&&\\
&\varpi^{-1} && \\
&&\varpi & \\
&&&1
\end{bmatrix} , \quad
t_n = 
\begin{bmatrix}
&&&-\varpi^{-n}\\
&1&& \\
&&1& \\
\varpi^n &&& 
\end{bmatrix}.
\end{equation} 
Sometimes we will  write $\eta$ and $\tau$ for $\pi(\eta)$ and $\pi(\tau)$, respectively. 
We define the level raising operators $\eta: V(n) \to V(n+2)$ and $\theta,\theta':V(n)\to V(n+1)$ as in \cite{RS}.
Let $(\pi,V)$ be an irreducible, admissible representation of $\GSp(4,F)$ with trivial central character. If $V(n)$ is non-zero
for some non-negative integer $n$ then we say that $\pi$ is \emph{paramodular} and let $N_\pi$ be the smallest such integer. It is known that if $\pi$ is paramodular, then
$V(N_\pi)$ is one-dimensional; if $W_\pi$ is a non-zero element of $V(N_\pi)$ so that $V(N_\pi) = \C \cdot W_\pi$, then we refer to $W_\pi$ as a \emph{newform}.
The space $V(n)$ for $n \geq N_\pi$ is spanned by the vectors $\theta'{}^i \theta^j \eta^k W_\pi$  where $i,j$ and $k$ are non-negative
integers with $i+j+2k = n-N_\pi$. 
It is known that if $\pi$ is generic, then $\pi$ is paramodular; in general, all paramodular, irreducible, admissible representations of $\GSp(4,F)$ with trivial central character have been classified. 
If $\pi$ is a generic, irreducible, admissible representation of $\GSp(4,F)$ with trivial central character then we let $\mathcal{W}(\pi,\psi_{c_1,c_2})$ be the Whittaker model of $\pi$ with respect to the character $\psi_{c_1,c_2}$ of the unipotent radical of the Borel subgroup of $\GSp(4,F)$ with $c_1,c_2 \in \OF^\times$. If $W_\pi$ is viewed as an element of the Whittaker model $\mathcal{W}(\pi,\psi_{c_1,c_2})$ of $\pi$, then $W_\pi(1) \neq 0$. 

We will refer to the following basic lemma. 
\begin{lemma}
\label{changelemma}
Let $X$ be a complex vector space.
Let $f: \OF^\times \to X$ be a locally constant function. Let $b \in \OF$ and let $t$ be a positive integer. We have
$$
\int\limits_{\OF^\times} f(u(1+bu^{-1} \varpi^t)) \, d u = \int\limits_{\OF^\times} f(u) \, d u. 
$$
\end{lemma}
\begin{proof} 
Let $n$ be a positive integer such that $f(x+\p^n) =f(x)$ for $x \in \OF^\times$. We have
\begin{align*}
\int\limits_{\OF^\times} f(u(1+bu^{-1} \varpi^t)) \, d u
&=
q^{-n} \sum_{u \in \OF^\times/(1+\p^n)}  f(u(1+bu^{-1}\varpi^t)). 
\end{align*}
and similarly
$$
\int\limits_{\OF^\times} f(u) \, d u= q^{-n} \sum_{u\in \OF^\times/(1+\p^n)}  f(u). 
$$
The lemma now follows from the fact that the function $\OF^\times/(1+\p^n) \to \OF^\times/(1+\p^n)$ defined by $u \mapsto u(1+bu^{-1}\varpi^t)$ is a well-defined bijection. 
\end{proof}
The following lemma about Gauss sums is well-known. 
\begin{lemma}
\label{gausslemma}
Let $\chi$ be a character of $\OF^\times$ with conductor $c(\chi)$, and let $k$ be an integer. Define
$$
G(\chi,k) = \int\limits_{\OF^\times} \chi(u)\psi(u \varpi^k)\, du.
$$
If $\chi$ is ramified, then $G(\chi,k)$ is non-zero if and only if $k=-c(\chi)$. 
\end{lemma}
\section{Twist in genus 1}
\label{genus1sec}
Let $(\pi,V)$ be a smooth representation of $\GL(2,F)$ for which the center of $\GL(2,F)$ acts trivially,
let $\chi$ be a quadratic character of $\OF^\times$ with conductor $c(\chi)$, and let $n$ be a non-negative integer. For $v \in V(n)$ we define
\begin{equation}
\label{gl2twisingdefeq}
T_\chi(v) = \int\limits_{\OF^\times} \chi(b) \pi(\begin{bmatrix} 1 & b\varpi^{-c(\chi)}\\ & 1 \end{bmatrix}) v\, db.
\end{equation}
If $\chi$ is unramified, then $T_\chi(v) = (1-q^{-1}) v$ for $v \in V(n)$. Thus, we will usually assume that $\chi$ is ramified. 
Assume further that $\pi$ is generic, irreducible and admissible with Whittaker model $\mathcal{W}(\pi,\psi)$. For $W \in \mathcal{W}(\pi,\psi)$
we define the $\chi$-twisted zeta integral of $W$ as
\begin{equation}
\label{gl2twistedzetaeq}
Z(s,W,\chi) = \int\limits_{F^\times} W(\begin{bmatrix} t & \\ & 1 \end{bmatrix}) |t|^{s-1/2} \chi(t)\, d^\times t. 
\end{equation}
\begin{theorem}
\label{gl2twisttheorem}
Let $(\pi,V)$ be a smooth representation of $\GL(2,F)$ for which the center of $\GL(2,F)$ acts trivially,
let $\chi$ be a quadratic character of $\OF^\times$ with conductor $c(\chi) > 0$, and let $n$ be a non-negative integer.
Let $N = \max(n,2 c(\chi))$. If $v \in V(n)$, then $\pi(k) T_\chi(v) = \chi (\det (k)) T_\chi (v)$ for $k \in \Gamma_0(\p^N)$.
Moreover, assume that $\pi$ is generic, irreducible and admissible with Whittaker model $\mathcal{W}(\pi,\psi)$. Let $W \in V(n)$.
The $\chi$-twisted zeta integral of $T_\chi(W)$ is
$$
Z(s,T_\chi(W),\chi) = (1-q^{-1})G(\chi,-c(\chi))W(1). 
$$
For $n \geq N_\pi$, the image of $T_\chi: V(n) \to V(N,\chi)$ is spanned by the non-zero vector $T_\chi( \beta'{}^{n-N_\pi} W_\pi)$. 
\end{theorem}
\begin{proof}
The group $\Gamma_0(\p^N)$ is generated by the elements contained in the sets
$$
\begin{bmatrix}
\OF^\times & \\ & \OF^\times 
\end{bmatrix}, \quad
\begin{bmatrix}
1& \OF \\
& 1 
\end{bmatrix}, \quad
\begin{bmatrix}
1& \\ 
\p^N&1 
\end{bmatrix}.
$$
It is easy to verify that $\pi(k)T_\chi(v) = \chi(\det(k))T_\chi(v)$ for generators $k$ of the first two types. 
Let $y \in \p^N$. Noting that $N-c(\chi) \geq c(\chi)>0$ and $N \geq n$, we have
\begin{align*}&\pi(\mat{1}{}{y}{1})
\int\limits_{\OF^\times} \chi(b)  \pi(\mat{1}{b\varpi^{-c(\chi)}}{}{1})v\, db\\
&=\int\limits_{\OF^\times} \chi(b)
\pi(\mat{(1+\varpi^{-c(\chi)}yb)^{-1}}{b\varpi^{-c(\chi)}}{}{1+\varpi^{-c(\chi)}y b})
\pi(\mat{1}{}{(1+\varpi^{-c(\chi)}yb)^{-1} y }{1})v\, db\\
&= \int\limits_{\OF^\times} \chi(b)
\pi(\mat{1}{(1+\varpi^{-c(\chi)}yb)^{-1}b\varpi^{-c(\chi)}}{}{1})
\pi(\mat{(1+\varpi^{-c(\chi)}yb)^{-1}}{}{}{1+\varpi^{-c(\chi)}y b})
v\, db\\
&= \int\limits_{\OF^\times} \chi(b)
\pi(\mat{1}{(1+\varpi^{-c(\chi)}yb)^{-1}b\varpi^{-c(\chi)}}{}{1})
v\, db\\
&= \int\limits_{\OF^\times} \chi((1+\varpi^{-c(\chi)}yb)^{-1}b)
\pi(\mat{1}{(1+\varpi^{-c(\chi)}yb)^{-1}b\varpi^{-c(\chi)}}{}{1})
v\, db\\
&= \int\limits_{\OF^\times} \chi(b)
\pi(\mat{1}{b\varpi^{-c(\chi)}}{}{1})
v\, db\\
&=T_{\chi}(v). 
\end{align*}
For the penultimate equality we applied Lemma \ref{changelemma}. 
Assume now that $\pi$ is generic, irreducible and admissible as in the statement of the theorem. Then:
\begin{align*}
Z (s,  T_{\chi} (W),\chi)
&=\int\limits_{F^\times}  T_{\chi}(W)(\mat{t}{}{}{1})|t|^{s-1/2}\chi(t)\, d^\times t\\
&=\int\limits_{F^\times} 
\int\limits_{\OF^\times} \chi(b) 
W(\mat{t}{}{}{1}\mat{1}{b\varpi^{-c(\chi)}}{}{1}) |t|^{s-1/2}\chi(t)\,  db\, d^\times t\\
&=\int\limits_{F^\times} 
(\int\limits_{\OF^\times} \chi(b) \psi(tb\varpi^{-c(\chi)}) \, db )
W(\mat{t}{}{}{1}) |t|^{s-1/2}\chi(t)\,  d^\times t\\
&=\int\limits_{\OF^\times} 
(\int\limits_{\OF^\times} \chi(b) \psi(tb\varpi^{-c(\chi)}) \, db )
W(\mat{t}{}{}{1}) \chi(t)\,  d^\times t\\
&=\int\limits_{\OF^\times} \chi(t) G(\chi,-c(\chi))
W(\mat{t}{}{}{1}) \chi(t)\,  d^\times t\\
&=(1-q^{-1})  G(\chi,-c(\chi))W(1). 
\end{align*}
Here, we have used Lemma \ref{gausslemma}. 
\end{proof}

\section{Twist in genus 2}
\label{genus2sec}
Let $(\pi, V)$ be a smooth representation of $\GSp(4,F)$ for which the center of $\GSp(4,F)$ acts trivially. Let $\chi$ be a quadratic character.  For $v \in V$ we define
\begin{equation}
\label{vprimedefalteq}
v^\chi = \int\limits_{\OF^\times} \int\limits_{\OF^\times} \int\limits_{\p^{-2c(\chi)}}  \chi(ab) 
\pi(
\begin{bmatrix} 1 & -a\varpi^{-c(\chi)} & b \varpi^{-2c(\chi)} & z\\ &1&&b\varpi^{-2c(\chi)} \\ &&1&a\varpi^{-c(\chi)} \\ &&&1 \end{bmatrix}
)\tau^{c(\chi)} v\,dz\, d a \, d b. 
\end{equation}
Evidently, if $\chi$ is unramified, then $v^\chi = (1-q^{-1})^2 v$. 
\begin{lemma}
\label{vprimelemma} 
Let $(\pi, V)$ be a smooth representation of $\GSp(4,F)$ for which the center of $\GSp(4,F)$ acts trivially. Let $\chi$ be a quadratic character with $c(\chi)>0$. Let $n$ be a non-negative integer. Let $v \in V_{\mathrm{Kl}}(n)$. We have
$
\pi(k)v^\chi =\chi(\lambda (k))  v^\chi
$
for the subgroup of $k \in \GSp(4,F)$ such that $\lambda (k) \in \OF^\times$ and 
$$
k \in 
\begin{bmatrix} 
\OF & \OF & \OF & \p^{-2c(\chi)} \\
&\OF&\OF&\OF\\
&\p^{2c(\chi)}&\OF&\OF\\
&&&\OF
\end{bmatrix}. 
$$ 
\end{lemma}
\begin{proof}
The subgroup in the statement of the lemma is generated by the elements of the form
$\diag( w_1w_2w, w_1w, w_2w, w )$
for $w,w_1,w_2 \in \OF^\times$, and elements of the subgroups
$$
\begin{bmatrix} 1&&&\p^{-2c(\chi)} \\ &1&& \\ &&1& \\ &&&1 \end{bmatrix}, \quad
\begin{bmatrix} 1&\OF&&\\ &1&& \\ &&1&\OF \\ &&&1 \end{bmatrix}, \quad
\begin{bmatrix} 1&&\OF&\\ &1&&\OF \\ &&1& \\ &&&1 \end{bmatrix}, \quad
\begin{bmatrix} 1&&& \\ &1&\OF& \\ &&1& \\ &&&1 \end{bmatrix}, \quad
\begin{bmatrix} 1&&& \\ &1&& \\ &\p^{2c(\chi)}&1& \\ &&&1 \end{bmatrix}.
$$
Using $v \in V_{\mathrm{Kl}}(n)$,  the definition of $v^\chi$  and Lemma \ref{changelemma} one can verify that
$\pi(k)v^\chi=\chi(\lambda(k))v^\chi$ for each type of  generator $k$. As an illustration, let $x \in \OF$. Then
\begin{align*}
&\pi(\begin{bmatrix} 1&-x&& \\ &1&& \\ &&1&x \\ &&&1 \end{bmatrix}) v^\chi \\
&= \int\limits_{\OF^\times} \int\limits_{\OF^\times} \int\limits_{\p^{-2c(\chi)}} 
\chi(ab) 
\pi(
\begin{bmatrix} 1 & -a\varpi^{-c(\chi)}-x & b \varpi^{-2c(\chi)} & z-xb\varpi^{-2c(\chi)} \\ &1&&b\varpi^{-2c(\chi)} \\ &&1&a\varpi^{-c(\chi)}+x \\ &&&1 \end{bmatrix})
 \tau^{c(\chi)} v\, dz\, d a \, d b\\
&=  \int\limits_{\OF^\times} \int\limits_{\OF^\times} \int\limits_{\p^{-2c(\chi)}}  \chi(ab) 
\pi(
\begin{bmatrix} 1 & -a(1+a^{-1}x\varpi^{c(\chi)})\varpi^{-c(\chi)} & b \varpi^{-2c(\chi)} & z \\ &1&&b\varpi^{-2c(\chi)} \\ &&1&a(1+a^{-1}x\varpi^{c(\chi)})\varpi^{-c(\chi)} \\ &&&1 \end{bmatrix})\\ 
&\qquad \tau^{c(\chi)} v\, dz\, d a \, d b\\
&=v^\chi. 
\end{align*}
This completes the proof. 
\end{proof}
\begin{lemma}
\label{firstnlemma}
Let $(\pi, V)$ be a smooth representation of $\GSp(4,F)$ for which the center of $\GSp(4,F)$ acts trivially. Let $\chi$ be a quadratic character and let $n$ be a non-negative integer. Let $v \in V_{\mathrm{Kl}}(n)$  and define $v^\chi$ as in \eqref{vprimedefalteq}. Then $v^\chi$ is invariant under
the subgroup 
$$
\GSp(4,F) \cap \begin{bmatrix} 1 &&& \\ \p^{\max(n+c(\chi),3c(\chi))}&1&& \\ \p^{\max(n+c(\chi),3c(\chi))} &&1& \\ \p^{\max(n+2c(\chi),4c(\chi))} &\p^{\max(n+c(\chi),3c(\chi))} &\p^{\max(n+c(\chi),3c(\chi))} &1 \end{bmatrix}.
$$
\end{lemma}
\begin{proof} 
This is clear if $\chi$ is unramified; assume that $c(\chi)>0$. 
Let $a,b \in \OF^\times$ and $c\in \OF$, and set
$$
g=
\begin{bmatrix} 1 & -a\varpi^{-c(\chi)} & b \varpi^{-2c(\chi)} & c\varpi^{-2c(\chi)}\\ &1&&b\varpi^{-2c(\chi)} \\ &&1&a\varpi^{-c(\chi)} \\ &&&1 \end{bmatrix} 
\tau^{c(\chi)}.
$$
Let $L$ be an integer and $y\in \OF$. We have the following identities:
\begin{align*}
&\begin{multlined}[t] \begin{bmatrix} 1 &&& \\ y\varpi^L&1&& \\ &&1& \\ && -y\varpi^L &1 \end{bmatrix} g\\
 = g\begin{bmatrix}
1+ay\varpi^{L-c(\chi)} & -a^2 y \varpi^{L-3c(\chi)} & (ab+c\varpi^{c(\chi)})y\varpi^{L-2c(\chi)} & 2acy\varpi^{L-3c(\chi)} \\
y\varpi^{L+c(\chi)} & 1-ay\varpi^{L-c(\chi)} & 2by \varpi^L & (ab+c\varpi^{c(\chi)})y\varpi^{L-2c(\chi)}\\
&&1+ay\varpi^{L-c(\chi)} & a^2 y \varpi^{L-3c(\chi)} \\
&&-y\varpi^{L+c(\chi)}  &  1-ay\varpi^{L-c(\chi)}
\end{bmatrix},\end{multlined}\\
&\begin{multlined}[t]\begin{bmatrix}
1&&&\\
&1&&\\
&&1&\\
y\varpi^L&&&1
\end{bmatrix}g \\
 = g \begin{bmatrix}
1-cy\varpi^{L-2c(\chi)} & a cy\varpi^{L-4c(\chi)} & - b cy \varpi^{L-3c(\chi)} & -c^2 y \varpi^{L-4c(\chi)}\\
- b y \varpi^{L-c(\chi)} & 1+ aby \varpi^{L-3c(\chi)} & -b^2 y \varpi^{L-2c(\chi)} & -bcy \varpi^{L-3c(\chi)}\\
-a y\varpi^{L-2c(\chi)} & a^2 y \varpi^{L-4c(\chi)} & 1- ab y \varpi^{L-3c(\chi)} & -a cy\varpi^{L-4c(\chi)} \\
y\varpi^L & -ay\varpi^{L-2c(\chi)} & b y\varpi^{L-c(\chi)} & 1+ cy\varpi^{L-2c(\chi)}
\end{bmatrix}.\end{multlined}
\end{align*}
These identities prove that $v^\chi$ is invariant under the group
$$
\GSp(4,F) \cap \begin{bmatrix} 1 &&& \\ \p^{\max(n+c(\chi),3c(\chi))}&1&& \\  &&1& \\ \p^{\max(n+2c(\chi),4c(\chi))} & &\p^{\max(n+c(\chi),3c(\chi))} &1 \end{bmatrix}.
$$
To prove the remaining invariance, set $L= \max(n+c(\chi),3c(\chi))$. A calculation shows that
\begin{multline*}
\begin{bmatrix}
1&&&\\
&1&&\\
y\varpi^L&&1&\\
&y\varpi^L&&1
\end{bmatrix}g
=
g
\begin{bmatrix}
1&&& (-2byc +ab^3y^2\varpi^{L-3c(\chi)})u^{-2} \varpi^{L-4c(\chi)}\\
&1&& \\ 
&&1& \\
&&&1
\end{bmatrix}\\
\begin{bmatrix}
1&abyu^{-1} \varpi^{L-4c(\chi)}&& \\ &1&& \\ &&1&-abyu^{-1} \varpi^{L-4c(\chi)} \\&&&1 
\end{bmatrix}k
\end{multline*}
for some $k \in \Kl{n}$ with $u = 1+by\varpi^{L-2c(\chi)}$.  Therefore,
\begin{align*}
&\pi(\begin{bmatrix}
1&&&\\
&1&&\\
y\varpi^L&&1&\\
&y\varpi^L&&1
\end{bmatrix}) v^\chi 
= q^{2c(\chi)} \int\limits_{\OF^\times} \int\limits_{\OF^\times} \int\limits_{\OF}  \chi(ab) 
\pi(
\begin{bmatrix} 1 & -a\varpi^{-c(\chi)} & b \varpi^{-2c(\chi)} & c\varpi^{-2c(\chi)}\\ &1&&b\varpi^{-2c(\chi)} \\ &&1&a\varpi^{-c(\chi)} \\ &&&1 \end{bmatrix}
)\tau^{c(\chi)} \\
&\begin{multlined}[t]\pi( \begin{bmatrix}
1&&& (-2byc +ab^3y^2\varpi^{L-3c(\chi)})u^{-2} \varpi^{L-4c(\chi)}\\
&1&& \\ 
&&1& \\
&&&1
\end{bmatrix}\\
\begin{bmatrix}
1&abyu^{-1} \varpi^{L-4c(\chi)}&& \\ &1&& \\ &&1&-abyu^{-1} \varpi^{L-4c(\chi)} \\&&&1 
\end{bmatrix})
 v\,dz\, d a \, d b\end{multlined}\\
&= q^{2c(\chi)}\int\limits_{\OF^\times} \int\limits_{\OF^\times} \int\limits_{\OF}  \chi(ab)  \\
&\pi(
\begin{bmatrix} 1 & -a( 1-byu^{-1}\varpi^{L-2c(\chi)}  )\varpi^{-c(\chi)} & b \varpi^{-2c(\chi)} & c'\varpi^{-2c(\chi)}\\ &1&&b\varpi^{-2c(\chi)} \\ &&1&a( 1-byu^{-1}\varpi^{L-2c(\chi)}  )\varpi^{-c(\chi)}  \\ &&&1 \end{bmatrix}
) \tau^{c(\chi)}
 v\,dz\, d a \, d b,
\end{align*}
where $c'=c(1-2byu^{-2}\varpi^{L-2c(\chi)})-ab^2yu^{-1}(1+byu^{-1}\varpi^{L-2c(\chi)})\varpi^{L-3c(\chi)}$.
Changing variables, we obtain $v^\chi$. 
\end{proof}
Let $(\pi,V)$ be a smooth representation of $\GSp(4,F)$ for which the center of $\GSp(4,F)$ acts trivially, let $\chi$ be a quadratic character, and let $n$ be a non-negative integer. For $v \in V_{\mathrm{Kl}}(n)$ we define 
\begin{equation}
\label{TKLdefeq}
T^{\mathrm{Kl}}_\chi(v)
=
\int\limits_{\OF} 
\pi(
\begin{bmatrix}
1&&&\\
&1&&\\
&x&1&\\
&&&1
\end{bmatrix}) v^\chi\, dx
+
\int\limits_{\p}
\pi(
\begin{bmatrix}
1&&&\\
&&1&\\
&-1&&\\
&&&1
\end{bmatrix}
\begin{bmatrix}
1&&&\\
&1&&\\
&y&1&\\
&&&1
\end{bmatrix})v^\chi\, dy.
\end{equation}
Here, $v^\chi$ as in \eqref{vprimedefalteq}. If $\chi$ is unramified, then $T_\chi^{\mathrm{Kl}}(v) = (1+q^{-1})(1-q^{-1})^2 v$. 
\begin{lemma}
\label{klingentranslemma}
Let $(\pi,V)$ be a smooth representation of $\GSp(4,F)$ for which the center of $\GSp(4,F)$ acts trivially, let $\chi$ be a quadratic character, and let $n$ be a non-negative integer. Let $v \in V_{\mathrm{Kl}}(n)$. Then 
\begin{equation}
\label{kltranseq}
 \pi(k) T_\chi^{\mathrm{Kl}}(v) = \chi(\lambda (k)) T_\chi^{\mathrm{Kl}} (v)
\end{equation}
for $k$ in $\mathrm{Kl}(\p^N)$ where $N = \max(n+2c(\chi),4c(\chi))$. Moreover, $\pi(k) T_\chi^{\mathrm{Kl}}(v) = T_\chi^{\mathrm{Kl}}(v)$ for $k \in \GSp(4,F)$ such that
$$
k \in \begin{bmatrix} 1 &&& \\ \p^{N-c(\chi)}&1&& \\ \p^{N-c(\chi)} &&1& \\ & \p^{N-c(\chi)} & \p^{N-c(\chi)} &1 \end{bmatrix}. 
$$
\end{lemma}
\begin{proof}
The group $\mathrm{Kl}(\p^N)$ is generated by its elements contained in the sets
$$
\begin{bmatrix} 1&&& \\ \p^N&1 & & \\ \p^N&&1& \\ \p^N&\p^N&\p^N & 1 \end{bmatrix}, \quad
\begin{bmatrix} 1&\OF & \OF & \OF \\ &1&& \OF \\ & &1& \OF \\ &&&1 \end{bmatrix}, \quad
\begin{bmatrix} \OF^\times &&& \\ &\OF & \OF & \\ & \OF & \OF & \\ &&&\OF^\times \end{bmatrix}, 
$$
and there is a disjoint decomposition 
$$
\SL(2,\OF) = \bigsqcup_{x \in \OF/\p^{2c(\chi)} } \begin{bmatrix} 1& \\ x & 1 \end{bmatrix} \Gamma_0(\p^{2c(\chi)}) \sqcup 
\bigsqcup_{y \in \p / \p^{2c(\chi)}   }\begin{bmatrix} &1 \\ -1 & \end{bmatrix} \begin{bmatrix} 1 & \\  y & 1 \end{bmatrix} \Gamma_0(\p^{2c(\chi)}).
$$
The lemma follows from these two facts, Lemma \ref{vprimelemma}, and Lemma  \ref{firstnlemma}. 
\end{proof}
Let $(\pi,V)$ be a generic, irreducible, admissible representation of $\GSp(4,F)$ with trivial central character
with Whittaker model $\mathcal{W}(\pi,\psi_{c_1,c_2})$; we take $c_1,c_2 \in \OF^\times$.
Let $\chi$ be a quadratic character of $F^\times$. If $W \in \mathcal{W}(\pi,\psi_{c_1,c_2})$
we define the zeta integral of $W$ twisted by $\chi$ to be
\begin{equation}
\label{twistedzetaeq}
Z(s,W,\chi) = \int\limits_{F^\times}\int\limits_{F} W(\begin{bmatrix} t&&& \\ &t&& \\&z&1& \\ &&&1 \end{bmatrix}) |t|^{s-3/2} \chi(t) \, dz\, d^\times t.
\end{equation}
This is the same as the zeta integral of $W$ in the twist $\chi \otimes \pi$ of $\pi$ by $\chi$. See \cite{RS}.
\begin{lemma}
\label{zetalemma}
Let $(\pi,V)$ be a generic, irreducible, admissible representation of $\GSp(4,F)$ with trivial central character
with Whittaker model $\mathcal{W}(\pi,\psi_{c_1,c_2})$; we take $c_1,c_2 \in \OF^\times$.
Let $\chi$ be a quadratic character of $F^\times$ such that $c(\chi)> 0$. 
Let $n$ be a non-negative integer.
Let $W \in V_{\mathrm{Kl}}(n)$, and define $T_\chi^{\mathrm{Kl}}(W)$ as in \eqref{TKLdefeq}. We have
\begin{equation}
\label{tklzetaeq}
Z(s, T_\chi^{\mathrm{Kl}}(W),\chi) = 
(1-q^{-1})  q^{c(\chi)} \chi(c_2)G(\chi,-c(\chi))^3 W(1 ). 
\end{equation}
In particular, if $W$ is the newform of $\pi$, then $T_\chi^{\mathrm{Kl}}(W) \neq 0$. 
\end{lemma}
\begin{proof}
To begin, we note that by Lemma 4.1.1 of \cite{RS} we have
$$
Z(s,T_\chi^{\mathrm{Kl}}(W),\chi) = \int\limits_{F^\times} T_\chi^{\mathrm{Kl}}(W)(\begin{bmatrix} t&&& \\ &t&& \\&&1& \\ &&&1 \end{bmatrix}) |t|^{s-3/2} \chi(t) \, d^\times t.
$$
Therefore, the first part of $Z(s,T_\chi^{\mathrm{Kl}}(W),\chi)$ is:
\begin{align*}
&\begin{multlined}[t]\int\limits_{F^\times}
\int\limits_{\OF} \int\limits_{\OF^\times} \int\limits_{\OF^\times} \int\limits_{\p^{-2c(\chi)}}
W(
\begin{bmatrix}
t&&& \\ &t&& \\ &&1& \\ &&&1 
\end{bmatrix}
\begin{bmatrix}
1&&&\\
&1&&\\
&x&1&\\
&&&1
\end{bmatrix} \\
\begin{bmatrix} 1 & -a\varpi^{-c(\chi)} & b \varpi^{-2c(\chi)} & z\\ &1&&b\varpi^{-2c(\chi)} \\ &&1&a\varpi^{-c(\chi)} \\ &&&1 \end{bmatrix}
\begin{bmatrix} 1&&& \\ &\varpi^{-c(\chi)}&& \\ &&\varpi^{c(\chi)}& \\ &&&1 \end{bmatrix} ) 
|t|^{s-3/2} \chi(t) \chi(ab)\, dz \, da\, db\, dx\, d^\times t \end{multlined}\\
&=\begin{multlined}[t]q^{2c(\chi)}\int\limits_{F^\times}
\int\limits_{\OF} \int\limits_{\OF^\times} \int\limits_{\OF^\times} 
W(
\begin{bmatrix}
t&&& \\ &t&& \\ &&1& \\ &&&1 
\end{bmatrix}
\begin{bmatrix}
1&&&\\
&1&&\\
&x&1&\\
&&&1
\end{bmatrix} \\
\begin{bmatrix} 1 & -a\varpi^{-c(\chi)} & b \varpi^{-2c(\chi)} & \\ &1&&b\varpi^{-2c(\chi)} \\ &&1&a\varpi^{-c(\chi)} \\ &&&1 \end{bmatrix}
\begin{bmatrix} 1&&& \\ &\varpi^{-c(\chi)}&& \\ &&\varpi^{c(\chi)}& \\ &&&1 \end{bmatrix} ) 
|t|^{s-3/2} \chi(t) \chi(ab)\,  da\, db\, dx\, d^\times t \end{multlined}\\
&=q^{2c(\chi)}\begin{multlined}[t]\int\limits_{F^\times}
\int\limits_{\OF} \int\limits_{\OF^\times} \int\limits_{\OF^\times} \psi(c_1(-a\varpi^{-c(\chi)}-bx\varpi^{-2c(\chi)})) 
W(
\begin{bmatrix}
t&&& \\ &t&& \\ &&1& \\ &&&1 
\end{bmatrix}
\begin{bmatrix}
1&&&\\
&1&&\\
&x&1&\\
&&&1
\end{bmatrix} \\
\begin{bmatrix} 1&&& \\ &\varpi^{-c(\chi)}&& \\ &&\varpi^{c(\chi)}& \\ &&&1 \end{bmatrix} ) 
|t|^{s-3/2} \chi(t) \chi(ab)\,  da\, db\, dx\, d^\times t \end{multlined}\\
&=q^{2c(\chi)}\begin{multlined}[t] ( \int\limits_{\OF^\times}\chi(a) \psi(-c_1 a\varpi^{-c(\chi)})\, da) 
 \int\limits_{\OF}
( \int\limits_{\OF^\times} \chi(b)  \psi(-c_1 bx\varpi^{-2c(\chi)}) \, db) \\
\int\limits_{F^\times}
W(
\begin{bmatrix}
t&&& \\ &t&& \\ &&1& \\ &&&1 
\end{bmatrix}
\begin{bmatrix}
1&&&\\
&1&&\\
&x&1&\\
&&&1
\end{bmatrix} 
\begin{bmatrix} 1&&& \\ &\varpi^{-c(\chi)}&& \\ &&\varpi^{c(\chi)}& \\ &&&1 \end{bmatrix} ) 
|t|^{s-3/2} \chi(t) \, d^\times t\, dx. \end{multlined}
\intertext{By Lemma \ref{gausslemma} the integral in the $b$ variable is zero unless $v(x) = c(\chi)$. Continuing,}
&=q^{c(\chi)}\begin{multlined}[t] ( \int\limits_{\OF^\times}\chi(a) \psi(-c_1 a\varpi^{-c(\chi)})\, da) 
 \int\limits_{\OF^\times}
( \int\limits_{\OF^\times} \chi(b)  \psi(-c_1 bx\varpi^{-c(\chi)}) \, db) \\
\int\limits_{F^\times}
W(
\begin{bmatrix}
t&&& \\ &t&& \\ &&1& \\ &&&1 
\end{bmatrix}
\begin{bmatrix}
1&&&\\
&1&&\\
&x\varpi^{c(\chi)}&1&\\
&&&1
\end{bmatrix} 
\begin{bmatrix} 1&&& \\ &\varpi^{-c(\chi)}&& \\ &&\varpi^{c(\chi)}& \\ &&&1 \end{bmatrix} ) 
|t|^{s-3/2} \chi(t) \, d^\times t\, dx \end{multlined}\\
&=\begin{multlined}[t] q^{c(\chi)}G(\chi,-c(\chi))^2
 \int\limits_{\OF^\times} \chi(x) 
\int\limits_{F^\times}
W(
\begin{bmatrix}
t&&& \\ &t&& \\ &&1& \\ &&&1 
\end{bmatrix}
\begin{bmatrix}
1&&&\\
&1&x^{-1}\varpi^{-c(\chi)}&\\
&&1&\\
&&&1
\end{bmatrix} \\
\begin{bmatrix}
1&&&\\
&-x^{-1}\varpi^{-c(\chi)}&&\\
&&-x\varpi^{c(\chi)}&\\
&&&1
\end{bmatrix}  
\begin{bmatrix}
1&&&\\
&&1&\\
&-1&&\\
&&&1
\end{bmatrix} 
\begin{bmatrix}
1&&&\\
&1&x^{-1}\varpi^{-c(\chi)}&\\
&&1&\\
&&&1
\end{bmatrix} \\
\begin{bmatrix} 1&&& \\ &\varpi^{-c(\chi)}&& \\ &&\varpi^{c(\chi)}& \\ &&&1 \end{bmatrix} ) 
|t|^{s-3/2} \chi(t) \, d^\times t\, dx \end{multlined}\\
&=\begin{multlined}[t] q^{c(\chi)}G(\chi,-c(\chi))^2
 \int\limits_{\OF^\times} \chi(x)
\int\limits_{F^\times}
\psi(c_2tx^{-1}\varpi^{-c(\chi)}) 
W(
\begin{bmatrix}
t&&& \\ &t&& \\ &&1& \\ &&&1 
\end{bmatrix}  ) 
|t|^{s-3/2} \chi(t) \, d^\times t\, dx \end{multlined}\\
&=\begin{multlined}[t] q^{c(\chi)}G(\chi,-c(\chi))^2  
\int\limits_{F^\times}
( \int\limits_{\OF^\times} \chi(x)
\psi(c_2tx\varpi^{-c(\chi)}) \, dx)
W(
\begin{bmatrix}
t&&& \\ &t&& \\ &&1& \\ &&&1 
\end{bmatrix}  ) 
|t|^{s-3/2} \chi(t) \, d^\times t\end{multlined}.
\intertext{Again, by Lemma \ref{gausslemma} the integral in the $x$ variable is zero unless $v(t)=0$. Thus, our quantity is:}
&=\begin{multlined}[t] q^{c(\chi)}G(\chi,-c(\chi))^2
\int\limits_{\OF^\times}
( \int\limits_{\OF^\times} \chi(x)
\psi(c_2tx\varpi^{-c(\chi)}) \, dx)
W(
\begin{bmatrix}
t&&& \\ &t&& \\ &&1& \\ &&&1 
\end{bmatrix}  ) 
|t|^{s-3/2} \chi(t) \, d^\times t\end{multlined} \\
&=\begin{multlined}[t] q^{c(\chi)}G(\chi,-c(\chi))^2
\int\limits_{\OF^\times}
\chi(t)( \int\limits_{\OF^\times} \chi(x)
\psi(c_2x\varpi^{-c(\chi)}) \, dx)
W(1 ) 
\chi(t) \, d^\times t\end{multlined}\\
&=\begin{multlined}[t] q^{c(\chi)} \chi(c_2)G(\chi,-c(\chi))^3
\int\limits_{\OF^\times}
W(1 ) 
 \, d^\times t\end{multlined}\\
&=\begin{multlined}[t] (1-q^{-1})  q^{c(\chi)} \chi(c_2)G(\chi,-c(\chi))^3 W(1 ).
\end{multlined}
\end{align*}
Finally, we prove that the second part of $Z(s,T_\chi^{\mathrm{Kl}}(W),\chi)$ is zero:
\begin{align*}
&\begin{multlined}[t]\int\limits_{F^\times}
\int\limits_{\p} \int\limits_{\OF^\times} \int\limits_{\OF^\times} \int\limits_{\p^{-2c(\chi)}}
W(
\begin{bmatrix}
t&&& \\ &t&& \\ &&1& \\ &&&1 
\end{bmatrix}
\begin{bmatrix}
1&&&\\
&&1&\\
&-1&&\\
&&&1
\end{bmatrix}
\begin{bmatrix}
1&&&\\
&1&&\\
&y&1&\\
&&&1
\end{bmatrix} \\
\begin{bmatrix} 1 & -a\varpi^{-c(\chi)} & b \varpi^{-2c(\chi)} & z\\ &1&&b\varpi^{-2c(\chi)} \\ &&1&a\varpi^{-c(\chi)} \\ &&&1 \end{bmatrix}
\begin{bmatrix} 1&&& \\ &\varpi^{-c(\chi)}&& \\ &&\varpi^{c(\chi)}& \\ &&&1 \end{bmatrix} ) 
|t|^{s-3/2} \chi(t) \chi(ab)\, dz \, da\, db\, dy\, d^\times t \end{multlined}\\
&=\begin{multlined}[t]q^{2c(\chi)}\int\limits_{F^\times}
\int\limits_{\p} \int\limits_{\OF^\times} \int\limits_{\OF^\times} 
\psi(-c_2ty) W(
\begin{bmatrix}
t&&& \\ &t&& \\ &&1& \\ &&&1 
\end{bmatrix}
\begin{bmatrix}
1&&&\\
&&1&\\
&-1&&\\
&&&1
\end{bmatrix} \\
\begin{bmatrix} 1 & -a\varpi^{-c(\chi)} & b \varpi^{-2c(\chi)} & \\ &1&&b\varpi^{-2c(\chi)} \\ &&1&a\varpi^{-c(\chi)} \\ &&&1 \end{bmatrix}
\begin{bmatrix} 1&&& \\ &\varpi^{-c(\chi)}&& \\ &&\varpi^{c(\chi)}& \\ &&&1 \end{bmatrix} ) 
|t|^{s-3/2} \chi(t) \chi(ab) \, da\, db\, dy\, d^\times t \end{multlined}\\
&=\begin{multlined}[t]q^{2c(\chi)}\int\limits_{F^\times}
\int\limits_{\p} \int\limits_{\OF^\times} \int\limits_{\OF^\times} 
\psi(-c_2ty)\psi(c_1b\varpi^{-2c(\chi)})  W(
\begin{bmatrix}
t&&& \\ &t&& \\ &&1& \\ &&&1 
\end{bmatrix}
\begin{bmatrix}
1&&&\\
&&1&\\
&-1&&\\
&&&1
\end{bmatrix} \\
\begin{bmatrix} 1&&& \\ &\varpi^{-c(\chi)}&& \\ &&\varpi^{c(\chi)}& \\ &&&1 \end{bmatrix} ) 
|t|^{s-3/2} \chi(t) \chi(ab) \, da\, db\, dy\, d^\times t \end{multlined}\\
&=0.
\end{align*}
The last equality holds because $\chi$ is ramified by assumption. 
\end{proof}
Let $(\pi,V)$ be a smooth representation of $\GSp(4,F)$ for which the center of $\GSp(4,F)$ acts trivially, let $\chi$ be a quadratic character, and let $n$ be a non-negative integer. Define $N=\max(n+2c(\chi),4c(\chi))$. For $v \in V_{\mathrm{Kl}}(n)$ we define 
\begin{equation}
\label{twisteq}
T_\chi(v) = 
q\int\limits_{\OF} \pi(\begin{bmatrix} 1&&& z \varpi^{-N} \\ &1&& \\ &&1& \\ &&&1 \end{bmatrix} ) T^{\mathrm{Kl}}_\chi (v)\, dz 
+
 \pi(t_N)  \int\limits_{\OF} \pi(\begin{bmatrix} 1&&& z \varpi^{-N+1} \\ &1&& \\ &&1& \\ &&&1 \end{bmatrix} ) T^{\mathrm{Kl}}_\chi (v)\, dz.
\end{equation}
Here, $T^{\mathrm{Kl}}_\chi (v)$ is defined as in \eqref{TKLdefeq}. Explicitly, 
\begin{align}
&q^{-2c(\chi)}T_\chi (v)\nonumber \\
&=q\int\limits_{\OF}\int\limits_{\OF}\int\limits_{\OF^\times}\int\limits_{\OF^\times}
\chi(ab)\pi(
\begin{bmatrix} 1&&& \\ &1&& \\ &x&1& \\ &&&1 \end{bmatrix} 
\begin{bmatrix} 1&-a\varpi^{-c(\chi)} & b\varpi^{-2c(\chi)} & z\varpi^{-N} \\ &1&&b\varpi^{-2c(\chi)} \\ &&1&a\varpi^{-c(\chi)} \\ &&&1 \end{bmatrix} )
\tau^{c(\chi)}v\, da\, db\, dx\, dz\label{termoneeq} \\
&+q\int\limits_{\OF}\int\limits_{\p}\int\limits_{\OF^\times}\int\limits_{\OF^\times} \begin{multlined}[t] 
\chi(ab)\pi(
\begin{bmatrix} 1&&& \\ &&1& \\ &-1&& \\ &&&1 \end{bmatrix}
\begin{bmatrix} 1&&& \\ &1&& \\ &y&1& \\ &&&1 \end{bmatrix}  \\
\begin{bmatrix} 1&-a\varpi^{-c(\chi)} & b\varpi^{-2c(\chi)} & z\varpi^{-N} \\ &1&&b\varpi^{-2c(\chi)} \\ &&1&a\varpi^{-c(\chi)} \\ &&&1 \end{bmatrix}\! )
\tau^{c(\chi)}v\, da\, db\, dy\, dz\end{multlined}\label{termtwoeq} \\
&+\int\limits_{\OF}\int\limits_{\OF}\int\limits_{\OF^\times}\int\limits_{\OF^\times} \chi(ab)
\pi(t_N
\begin{bmatrix} 1&&& \\ &1&& \\ &x&1& \\ &&&1 \end{bmatrix} 
\begin{bmatrix} 1&-a\varpi^{-c(\chi)} & b\varpi^{-2c(\chi)} & z\varpi^{-N+1} \\ &1&&b\varpi^{-2c(\chi)} \\ &&1&a\varpi^{-c(\chi)} \\ &&&1 \end{bmatrix} )
\tau^{c(\chi)}v\, da\, db\, dx\, dz \label{termthreeeq} \\
&+\int\limits_{\OF}\!\int\limits_{\p}\!\int\limits_{\OF^\times} \int\limits_{\OF^\times} \begin{multlined}[t]
\pi(t_N 
\begin{bmatrix} 1&&& \\ &&1& \\ &-1&& \\ &&&1 \end{bmatrix} 
\begin{bmatrix} 1&&& \\ &1&& \\ &y&1& \\ &&&1 \end{bmatrix}\\
\begin{bmatrix} 1&-a\varpi^{-c(\chi)} & b\varpi^{-2c(\chi)} & z\varpi^{-N+1} \\ &1&&b\varpi^{-2c(\chi)} \\ &&1&a\varpi^{-c(\chi)} \\ &&&1 \end{bmatrix} )
\tau^{c(\chi)}v\, da\, db\, dy\, dz.\end{multlined} \label{termfoureq}
\end{align}

\begin{lemma}
\label{mainlemma}
Let $(\pi,V)$ be a smooth representation of $\GSp(4,F)$ for which the center of $\GSp(4,F)$ acts trivially and let $\chi$ be a quadratic character.  Let $n$ be a non-negative integer and define 
$
N= \max(n+2c(\chi),4c(\chi)).
$
Let $v \in V_{\mathrm{Kl}}(n)$. 
\begin{enumerate}
\item[i)] We have $\pi(k) T_\chi(v) = \chi(\lambda(k)) T_\chi (v)$  for $k \in \K{N}$.
\item[ii)] Assume that $c(\chi)>0$. If $T_\chi(v)$ is invariant under 
under the elements
\begin{equation}
\label{r1r2eq}
\begin{bmatrix} 1 &r_1 \varpi^{-1} & r_2 \varpi^{-1} & \\ &1&&r_2\varpi^{-1} \\ &&1& -r_1 \varpi^{-1} \\ &&& 1 \end{bmatrix}
\end{equation}
for $r_1,r_2 \in \OF$, then $T_\chi(v)=0$.
\end{enumerate}
\end{lemma}
\begin{proof}
i)  
Fix a Haar measure for the group $\GSp(4,F)$. By Lemma 3.3.1 of \cite{RS} there is a disjoint decomposition
$$
\K{N} = \bigsqcup_{z \in \OF/\p^N} \begin{bmatrix} 1&&&z \varpi^{-N} \\ &1&& \\ &&1& \\ &&&1 \end{bmatrix} \Kl{N}
\sqcup \bigsqcup_{z \in \OF/\p^{N-1}} t_N \begin{bmatrix} 1&&&z \varpi^{-N+1} \\ &1&& \\ &&1& \\ &&&1 \end{bmatrix} \Kl{N}.
$$
Here, the second disjoint union is not present if $N=0$. Therefore, by \eqref{kltranseq},
\begin{align*}
\int\limits_{\K{N}} \chi (\lambda (k)) \pi (k) T_\chi^{\mathrm{Kl}}(v)\, dk
&=\begin{multlined}[t] \vl (\Kl{N}) \sum_{z \in \OF/\p^N} 
 \pi (\begin{bmatrix} 1&&&z \varpi^{-N} \\ &1&& \\ &&1& \\ &&&1 \end{bmatrix} ) T_\chi^{\mathrm{Kl}}(v)\\
+ \vl (\Kl{N}) \sum_{z \in \OF/\p^{N-1}}  \pi (t_N \begin{bmatrix} 1&&&z \varpi^{-N+1} \\ &1&& \\ &&1& \\ &&&1 \end{bmatrix} ) T_\chi^{\mathrm{Kl}}(v) \end{multlined} \\
&=\begin{multlined}[t] \vl (\Kl{N}) q^N\int\limits_{\OF}
 \pi (\begin{bmatrix} 1&&&z \varpi^{-N} \\ &1&& \\ &&1& \\ &&&1 \end{bmatrix} ) T_\chi^{\mathrm{Kl}}(v)\, dz\\
+ \vl (\Kl{N})q^{N-1} \int\limits_{\OF}  \pi (t_N \begin{bmatrix} 1&&&z \varpi^{-N+1} \\ &1&& \\ &&1& \\ &&&1 \end{bmatrix} ) T_\chi^{\mathrm{Kl}}(v)\, dz. \end{multlined}
\end{align*}
This is a positive multiple of $T_\chi(v)$, and thus  implies the desired transformation rule.

ii)  Assume that $T_\chi(v)$ is invariant under the elements in \eqref{r1r2eq}. Then
\begin{align}
T_\chi(v)
& = \int\limits_{\OF}\int\limits_{\OF} \pi ( \begin{bmatrix} 1 &r_1 \varpi^{-1} & r_2 \varpi^{-1} & \\ &1&&r_2\varpi^{-1} \\ &&1& -r_1 \varpi^{-1} \\ &&& 1 \end{bmatrix} ) T_\chi(v) \, dr_1\, dr_2\nonumber\\
&=\begin{multlined}[t] q\int\limits_{\OF} \int\limits_{\OF}\int\limits_{\OF} \pi(\begin{bmatrix} 1 &r_1 \varpi^{-1} & r_2 \varpi^{-1} &z \varpi^{-N}  \\ &1&&r_2\varpi^{-1} \\ &&1& -r_1 \varpi^{-1} \\ &&& 1 \end{bmatrix}) T^{\mathrm{Kl}}_\chi (v)\, dr_1\, dr_2\, dz \\
+
 \pi(t_N)  \int\limits_{\OF}\int\limits_{\OF} \int\limits_{\OF} \pi(
\begin{bmatrix} 1&&& \\ r_2\varpi^{N-1} &1&& \\ -r_1\varpi^{N-1} &&1& \\ & -r_1\varpi^{N-1} & -r_2 \varpi^{N-1} & 1 \end{bmatrix} \\
\begin{bmatrix} 1&&& z \varpi^{-N+1} \\ &1&& \\ &&1& \\ &&&1 \end{bmatrix} ) T^{\mathrm{Kl}}_\chi (v)\, dz\, dr_1\, dr_2. 
\end{multlined} \label{etaTchieq}
\end{align}
We claim that the first summand of \eqref{etaTchieq} is zero. Now
\begin{align*}
 &\int\limits_{\OF}\int\limits_{\OF}\int\limits_{\OF} \pi( \begin{bmatrix} 1 &r_1 \varpi^{-1} & r_2 \varpi^{-1} & z \varpi^{-N} \\ &1&&r_2\varpi^{-1} \\ &&1& -r_1 \varpi^{-1} \\ &&& 1 \end{bmatrix}) T^{\mathrm{Kl}}_\chi (v)\, dr_1\, dr_2\, dz\\
&\begin{multlined}[t] =\int\limits_{\OF}  \int\limits_{\OF}\int\limits_{\OF}\int\limits_{\OF}
\pi(
\begin{bmatrix}
1&&&\\
&1&&\\
&x&1&\\
&&&1
\end{bmatrix}
 \begin{bmatrix} 1 &(r_1+xr_2) \varpi^{-1} & r_2 \varpi^{-1} & z \varpi^{-N}\\ &1&&r_2\varpi^{-1} \\ &&1& -(r_1+xr_2) \varpi^{-1} \\ &&& 1 \end{bmatrix}
) v^\chi\,dr_1\, dr_2\,  dx\, dz\\
+
\int\limits_{\p} \int\limits_{\OF}\int\limits_{\OF}\int\limits_{\OF}
\pi(
\begin{bmatrix}
1&&&\\
&&1&\\
&-1&&\\
&&&1
\end{bmatrix} 
\begin{bmatrix}
1&&&\\
&1&&\\
&y&1&\\
&&&1
\end{bmatrix} \\
\begin{bmatrix} 1 &(r_1y-r_2) \varpi^{-1} & r_1 \varpi^{-1} &z\varpi^{-N} \\ &1&&r_1\varpi^{-1} \\ &&1& -(r_1y-r_2) \varpi^{-1} \\ &&& 1 \end{bmatrix}
)v^\chi\,dr_1\,dr_2\, dy\, dz\end{multlined} \\
&\begin{multlined}[t] =\int\limits_{\OF} \int\limits_{\OF} \int\limits_{\OF}  \int\limits_{\OF} 
\pi(
\begin{bmatrix}
1&&&\\
&1&&\\
&x&1&\\
&&&1
\end{bmatrix}
 \begin{bmatrix} 1 & r_1 \varpi^{-1} & r_2 \varpi^{-1} & z\varpi^{-N}\\ &1&&r_2\varpi^{-1} \\ &&1& -r_1 \varpi^{-1} \\ &&& 1 \end{bmatrix}
) v^\chi\,dr_1\, dr_2\,  dx\, dz\\
+
\int\limits_{\p}\int\limits_{\OF}\int\limits_{\OF}\int\limits_{\OF}
\pi(
\begin{bmatrix}
1&&&\\
&&1&\\
&-1&&\\
&&&1
\end{bmatrix} 
\begin{bmatrix}
1&&&\\
&1&&\\
&y&1&\\
&&&1
\end{bmatrix} 
\begin{bmatrix} 1 &r_2 \varpi^{-1} & r_1 \varpi^{-1} &z\varpi^{-N} \\ &1&&r_1\varpi^{-1} \\ &&1& -r_2 \varpi^{-1} \\ &&& 1 \end{bmatrix}
)v^\chi\,dr_1\,dr_2\, dy\, dz.\end{multlined}
\end{align*}
Moreover,
\begin{align*}
&\int\limits_{\OF} \int\limits_{\OF}  \int\limits_{\OF} 
\pi(
 \begin{bmatrix} 1 & r_1 \varpi^{-1} & r_2 \varpi^{-1} & z\varpi^{-N}\\ &1&&r_2\varpi^{-1} \\ &&1& -r_1 \varpi^{-1} \\ &&& 1 \end{bmatrix}
) v^\chi\,dr_1\, dr_2\, dz\\
&=q^{2c(\chi)} \int\limits_{\OF^\times} \int\limits_{\OF^\times}  \int\limits_{\OF} \int\limits_{\OF} \chi(ab) 
\pi( 
\begin{bmatrix} 1 & -au_1\varpi^{-c(\chi)} & bu_2 \varpi^{-2c(\chi)} & z\varpi^{-N}\\ &1&&bu_2 \varpi^{-2c(\chi)} \\ &&1& au_1 \varpi^{-c(\chi)} \\ &&&1 \end{bmatrix} ) v\,dr_1\, dr_2 \,dz\, d a \, d b  \\
\end{align*}
with 
$$
u_1 = 1-r_1a^{-1}\varpi^{c(\chi)-1} \quad \text{and}\quad u_2 = 1+b^{-1}r_2\varpi^{2c(\chi)-1}.
$$
Assume first $c(\chi)=1$. Then this integral  is:
\begin{align*}
& \int\limits_{\OF^\times} \int\limits_{\OF^\times}  \int\limits_{\OF} \int\limits_{\OF} \chi(ab) 
\pi( 
\begin{bmatrix} 1 & -(a-r_1)\varpi^{-c(\chi)} & bu_2 \varpi^{-2c(\chi)} & z\varpi^{-N}\\ &1&&bu_2 \varpi^{-2c(\chi)} \\ &&1& (a-r_1) \varpi^{-c(\chi)} \\ &&&1 \end{bmatrix} ) v\,dr_1\, dr_2 \,dz\, d a \, d b  \\
& \int\limits_{\OF^\times} \int\limits_{\OF^\times}  \int\limits_{\OF} \int\limits_{\OF} \chi(ab) 
\pi( 
\begin{bmatrix} 1 & r_1\varpi^{-c(\chi)} & bu_2 \varpi^{-2c(\chi)} & z\varpi^{-N}\\ &1&&bu_2 \varpi^{-2c(\chi)} \\ &&1& -r_1 \varpi^{-c(\chi)} \\ &&&1 \end{bmatrix} ) v\,dr_1\, dr_2 \,dz\, d a \, d b  \\
&=0.
\end{align*}
Assume that $c(\chi)>1$. Changing variables in $r_1$ and then in $a$, this integral is:
\begin{align*}
&\begin{multlined}[t] \int\limits_{\OF^\times} \int\limits_{\OF^\times}  \int\limits_{\OF} \int\limits_{\OF} \chi(ab) \\
\pi( 
\begin{bmatrix} 1 & -a(1+r_1\varpi^{c(\chi)-1}) \varpi^{-c(\chi)} & bu_2 \varpi^{-2c(\chi)} & z\varpi^{-N}\\ &1&&bu_2 \varpi^{-2c(\chi)} \\ &&1& a(1+r_1\varpi^{c(\chi)-1})  \varpi^{-c(\chi)} \\ &&&1 \end{bmatrix} ) v\,dr_1\, dr_2 \,dz\, d a \, d b \end{multlined} \\
&=\begin{multlined}[t]  \int\limits_{\OF^\times} \int\limits_{\OF^\times}  \int\limits_{\OF} \int\limits_{\OF} \chi(a(1+r_1\varpi^{c(\chi)-1})b) 
\pi( 
\begin{bmatrix} 1 & -a \varpi^{-c(\chi)} & bu_2 \varpi^{-2c(\chi)} & z\varpi^{-N}\\ &1&&bu_2 \varpi^{-2c(\chi)} \\ &&1& a  \varpi^{-c(\chi)} \\ &&&1 \end{bmatrix} ) v\,dr_1\, dr_2 \,dz\, d a \, d b \end{multlined} \\
&=0. 
\end{align*}
This proves that the first summand of \eqref{etaTchieq} is zero, as claimed. We now have:
\begin{multline*}
T_\chi(v)
 = 
 \pi(t_N)  \int\limits_{\OF}\int\limits_{\OF} \int\limits_{\OF} \pi(
\begin{bmatrix} 1&&& \\ r_2\varpi^{N-1} &1&& \\ -r_1\varpi^{N-1} &&1& \\ & -r_1\varpi^{N-1} & -r_2 \varpi^{N-1} & 1 \end{bmatrix} \\
\begin{bmatrix} 1&&& z \varpi^{-N+1} \\ &1&& \\ &&1& \\ &&&1 \end{bmatrix} ) T^{\mathrm{Kl}}_\chi (v)\, dz\, dr_1\, dr_2.
\end{multline*}
Applying $\pi(t_N)^{-1}$ to both sides and using the invariance of $T_\chi(v)$ under $t_N \in \K{N}$ from i), we see that $T_\chi(v)$ is: 
\begin{align*}
&\int\limits_{\OF}\int\limits_{\OF} \int\limits_{\OF} \pi(
\begin{bmatrix} 1&&& \\ r_2\varpi^{N-1} &1&& \\ -r_1\varpi^{N-1} &&1& \\ & -r_1\varpi^{N-1} & -r_2 \varpi^{N-1} & 1 \end{bmatrix} 
\begin{bmatrix} 1&&& z \varpi^{-N+1} \\ &1&& \\ &&1& \\ &&&1 \end{bmatrix} ) T^{\mathrm{Kl}}_\chi (v)\, dz\, dr_1\, dr_2\\
& = \int\limits_{\OF}\int\limits_{\OF} \int\limits_{\OF} \pi(
\begin{bmatrix} 1&&& z \varpi^{-N+1} \\ &1&& \\ &&1& \\ &&&1 \end{bmatrix} 
\begin{bmatrix} 1&r_1z&r_2z& \\ r_2\varpi^{N-1} &1&&r_2z \\ -r_1\varpi^{N-1} &&1&-r_1z \\ & -r_1\varpi^{N-1} & -r_2 \varpi^{N-1} & 1 \end{bmatrix} 
) T^{\mathrm{Kl}}_\chi (v)\, dz\, dr_1\, dr_2 \\
& = \int\limits_{\OF}\pi(
\begin{bmatrix} 1&&& z \varpi^{-N+1} \\ &1&& \\ &&1& \\ &&&1 \end{bmatrix} 
) T^{\mathrm{Kl}}_\chi (v)\, dz
\end{align*}
where we have used the invariance properties of $T^{\mathrm{Kl}}_\chi (v)$ from Lemma \ref{klingentranslemma}. By assumption, $T_\chi(v)$ is invariant under the elements of the form \eqref{r1r2eq}; integrating again over these elements we have
$$
T_\chi(v) =\int\limits_{\OF}\int\limits_{\OF} \int\limits_{\OF}\pi(
             \begin{bmatrix} 1 &r_1 \varpi^{-1} & r_2 \varpi^{-1} & z \varpi^{-N+1} \\ &1&&r_2\varpi^{-1} \\ &&1& -r_1 \varpi^{-1} \\ &&& 1 \end{bmatrix}
) T^{\mathrm{Kl}}_\chi (v)\, dr_1\, dr_2\, dz.
$$
This integral is zero by an argument analogous to the one above proving that the first term of \eqref{etaTchieq} is zero. The proof  is complete. 
\end{proof}

\begin{theorem}
\label{maintheorem}
Let $(\pi,V)$ be a smooth representation of $\GSp(4,F)$ for which the center of $\GSp(4,F)$ acts trivially, and let $\chi$ be a quadratic character of $F^\times$ with conductor $c(\chi) >0$.  Let $n$ be a non-negative integer and define 
$
N=  \max(n+2c(\chi),4c(\chi)).
$
 If $v \in V (n)$, then $T_\chi(v) \in V(N,\chi)$. 
Moreover, assume that $\pi$ is generic, irreducible and admissible with Whittaker model $\mathcal{W}(\pi,\psi_{c_1,c_2})$ where $c_1,c_2 \in \OF^\times$. If $W \in  V (n)$, then the $\chi$-twisted zeta integral \eqref{twistedzetaeq} of $T_\chi(W)$ is
\begin{equation}
\label{tchizetaeq}
Z(s, T_\chi(W),\chi) = (q-1)  q^{c(\chi)} \chi(c_2)G(\chi,-c(\chi))^3 W(1 ).
\end{equation}
For $n \geq N_\pi$, the image of $T_\chi: V(n) \to V(N,\chi)$ is spanned by the non-zero vector $T_\chi( \theta'{}^{n-N_\pi} W_\pi)$, 
where $W_\pi$ is a newform for $\pi$. 
\end{theorem}
\begin{proof}
The first assertion was proven in i) of Lemma \ref{mainlemma}. 
Assume now that $\pi$ is generic and irreducible. We work in the Whittaker model $ \mathcal{W}(\pi,\psi_{c_1,c_2})$ with $c_1,c_2 \in \OF^\times$. 
By Lemma 4.1.1 of \cite{RS} we have
$$
Z(s,T_\chi(v), \chi) =
\int\limits_{F^\times} T_\chi(v) (\begin{bmatrix} t&&& \\ &t&& \\ &&1& \\ &&&1 \end{bmatrix} ) |t|^{s-3/2} \chi(t) \, d^\times t.
$$
By the definition of $T_\chi(v)$, this is
\begin{multline*}
q \int\limits_{F^\times}  \int\limits_{\OF}  T^{\mathrm{Kl}}_\chi (v) 
(\begin{bmatrix} t&&& \\ &t&& \\ &&1& \\ &&&1 \end{bmatrix} 
 \begin{bmatrix} 1&&& z \varpi^{-N} \\ &1&& \\ &&1& \\ &&&1 \end{bmatrix} )|t|^{s-3/2} \chi(t) \, dz   \, d^\times t \\
+
\int\limits_{F^\times} \int\limits_{\OF} 
T^{\mathrm{Kl}}_\chi (v)(\begin{bmatrix} t&&& \\ &t&& \\ &&1& \\ &&&1 \end{bmatrix}  
t_N\begin{bmatrix} 1&&& z \varpi^{-N+1} \\ &1&& \\ &&1& \\ &&&1 \end{bmatrix}  )|t|^{s-3/2} \chi(t) \, dz\, d^\times t. 
\end{multline*}
We assert that the second summand is zero; it will suffice to prove that the integrand is zero. Let $t \in F^\times$ and $z \in \OF$. Let
$x \in \OF$. Then
\begin{align*}
&\psi(c_1x \varpi^{-1}) T^{\mathrm{Kl}}_\chi (v)(\begin{bmatrix} t&&& \\ &t&& \\ &&1& \\ &&&1 \end{bmatrix}  
t_N\begin{bmatrix} 1&&& z \varpi^{-N+1} \\ &1&& \\ &&1& \\ &&&1 \end{bmatrix}  ) \\
&= T^{\mathrm{Kl}}_\chi (v)( \begin{bmatrix} 1&x\varpi^{-1} && \\ &1&& \\ &&1&-x \varpi^{-1} \\ &&& 1 \end{bmatrix} 
\begin{bmatrix} t&&& \\ &t&& \\ &&1& \\ &&&1 \end{bmatrix}  
t_N\begin{bmatrix} 1&&& z \varpi^{-N+1} \\ &1&& \\ &&1& \\ &&&1 \end{bmatrix}  ) \\
&= \begin{multlined}[t] T^{\mathrm{Kl}}_\chi (v)( 
\begin{bmatrix} t&&& \\ &t&& \\ &&1& \\ &&&1 \end{bmatrix}  
t_N \begin{bmatrix} 1&&& z \varpi^{-N+1} \\ &1&& \\ &&1& \\ &&&1 \end{bmatrix} \\
\begin{bmatrix} 1&&& \\ &1&& \\ -x\varpi^{N-1}&&1& \\ &-x\varpi^{N-1}&& 1 \end{bmatrix} 
\begin{bmatrix} 1& xz && \\ &1&& \\ & x^2z\varpi^{N-1} &1 & -xz \\ &&&1 \end{bmatrix} )\end{multlined} \\
&= T^{\mathrm{Kl}}_\chi (v)( 
\begin{bmatrix} t&&& \\ &t&& \\ &&1& \\ &&&1 \end{bmatrix}  
t_N \begin{bmatrix} 1&&& z \varpi^{-N+1} \\ &1&& \\ &&1& \\ &&&1 \end{bmatrix}),
\end{align*}
where the last equality follows from the invariance properties of Lemma \ref{klingentranslemma}. Since $\psi(\p^{-1}) \neq 1$, this implies that the integrand is zero. The first summand is
\begin{align*}
&q \int\limits_{F^\times}  \int\limits_{\OF}  T^{\mathrm{Kl}}_\chi (v) 
(\begin{bmatrix} t&&& \\ &t&& \\ &&1& \\ &&&1 \end{bmatrix} 
 \begin{bmatrix} 1&&& z \varpi^{-N} \\ &1&& \\ &&1& \\ &&&1 \end{bmatrix} )|t|^{s-3/2} \chi(t) \, dz   \, d^\times t \\
&=q \int\limits_{F^\times}  \int\limits_{\OF}  T^{\mathrm{Kl}}_\chi (v) 
(\begin{bmatrix} t&&& \\ &t&& \\ &&1& \\ &&&1 \end{bmatrix} )|t|^{s-3/2} \chi(t) \, dz   \, d^\times t \\
&=q \int\limits_{F^\times}   T^{\mathrm{Kl}}_\chi (v) 
(\begin{bmatrix} t&&& \\ &t&& \\ &&1& \\ &&&1 \end{bmatrix} )|t|^{s-3/2} \chi(t)   \, d^\times t \\
&=qZ(s,T^{\mathrm{Kl}}_\chi (v), \chi).
\end{align*}
The formula \eqref{tchizetaeq} follows now from \eqref{tklzetaeq}. To prove the final assertion, we note first by Theorem 7.5.7 of \cite{RS} that the space $V(n)$ is spanned by the vectors $\theta'{}^i \theta^j \eta^k W_\pi$ with $i+j+2k=n-N_\pi$. The formula (3.7) of \cite{RS} implies that 
\begin{align*}
Z(s, T_\chi(\theta'{}^{n-N_\pi} W_\pi),\chi) &=  (q-1)  q^{c(\chi)} \chi(c_2)G(\chi,-c(\chi))^3 (\theta'{}^{n-N_\pi}W_\pi)(1 ) \\
&= (q-1)  q^{c(\chi)+n-N_\pi} \chi(c_2)G(\chi,-c(\chi))^3 W_\pi(1),
\end{align*}
and this is non-zero. To complete the proof, it will suffice to prove that $T_\chi(\theta'{}^i \theta^j \eta^k W_\pi)=0$ if $j>0$ or $k>0$. Let $W=\theta'{}^i \theta^j \eta^k W_\pi$ with $j>0$ or $k>0$. The $\chi$-twisted zeta integral of $W$ is a constant times $(\theta'{}^i \theta^j \eta^k W_\pi)(1)$; this quantity is zero by the definitions of $\eta$, $\theta$, and Lemma 4.1.2 of \cite{RS}. Since $Z(s,T_\chi(W),\chi)=0$,  by Theorem 4.3.7 of \cite{RS} there exists $W' \in V(N-2,\chi)$ such that $T_\chi(W) = \eta W'$. This implies that $T_\chi(W)$ is invariant under the elements in \eqref{r1r2eq}. Therefore, by ii) of Lemma \ref{mainlemma},  $T_\chi(W)=0$. 
\end{proof}

\end{document}